\documentclass[12pt]{amsart}

\usepackage[english]{babel}

\usepackage{amsmath}
\usepackage{amssymb}
\usepackage{amsfonts}
\usepackage{graphicx}

\newtheorem{theorem}{Theorem}[section]
\newtheorem{lemma}[theorem]{Lemma}

\theoremstyle{example}

\theoremstyle{definition}

\newtheorem{corollary}[theorem]{Corollary}

\newtheorem{ob}[theorem]{Notation}

\begin{document}

\thispagestyle{empty}

\title[]{A Hausdorff compact space is metrizable if and only if it is a continuous open image of the Sorgenfrey line}
\author{{Vlad Smolin}}
\address{Vlad Smolin
\newline\hphantom{iii} Krasovskii Institute of Mathematics and Mechanics,
\newline\hphantom{iii} Sofia Kovalevskaya street, 16,
\newline\hphantom{iii} 620990, Ekaterinburg, Russia}
\email{SVRusl@yandex.ru}

\maketitle{\small
\begin{quote}
\noindent{\sc Abstract. } In this note we prove that a regular continuous open image of the Sorgenfrey line with an uncountable weight has a closed subspace that is homeomorphic to the Sorgenfrey line. As a corollary we deduce the theorem in the title.
 \end{quote}
}

\section{Introduction}

A continuous map is called open if the image of an open set under this map is open.

In \cite{Sm} the author asked if there is a Hausdorff nonmetrizable compact space that is a continuous open image of the Sorgenfrey line? In this paper we give a negative answer to this question by proving the theorem in the title.

\section{Notation and terminology}

We use terminology from \cite{2}.

The symbol := means ``equals by definition''; the symbol $:\longleftrightarrow$ is used to show that the expression on the left side is an abbreviation for the expression on the right side.

\begin{ob} {\rm Let $\langle X, \tau \rangle$ be a topological space, $x \in X$, $B \subseteq X$, and $A \subseteq \mathbb{R}$. Then
    \begin{itemize}
        \item $\omega := $ the set of finite ordinals = the set of natural numbers;
        \item $f{\upharpoonright} A :=$ the restriction of function $f$ to $A$;
        \item $\mathbb{S} :=$ the Sorgenfrey line $:= \langle \mathbb{R}, \tau_{\mathbb{S}} \rangle$, where $\tau_{\mathbb{S}}$ is the topology generated by $\{[a, b): a,b \in \mathbb{R}\}$;
        \item $A_{\mathbb{S}} :=$ the set $A$ as a subspace of $\mathbb{S}$;
        \item $A_{\mathbb{R}} :=$ the set $A$ as a subspace of $\langle \mathbb{R}, \tau_{\mathbb{R}} \rangle$, where $\tau_{\mathbb{R}}$ is the natural topology on the real line;
        \item if $p \in {}^{\omega} X$, then $p \xrightarrow{\langle X, \tau \rangle} x :\longleftrightarrow p$ converges to $x$ in $\langle X, \tau \rangle$;
        \item $\mathsf{nbhds}(x, \tau) := \{U \in \tau: x \in U\}$;
        \item $\tau {\upharpoonright} B := \{U \cap B: U \in \tau\} = $ the subspace topology of $B$;
        \item $\mathsf{Cl}_{\langle X, \tau \rangle}(B) := $ the closure of $B$ in $\langle X, \tau \rangle$;
        \item if $\langle Y, \sigma \rangle$ is a topological space, then $\langle X, \tau \rangle \cong \langle Y, \sigma \rangle :\longleftrightarrow \langle X, \tau \rangle$ is homeomorphic to $\langle Y, \sigma \rangle$;
        \item $\mathsf{w}(\langle X, \tau\rangle)$ := the weight of $\langle X, \tau\rangle$.
    \end{itemize}
}\end{ob}

\section{Results}

\begin{lemma}
    Let $\langle X, \tau \rangle$ be a $T_1$ topological space, $f:\mathbb{S}\rightarrow \langle X, \tau \rangle$ a continuous open surjection, and $A$ at most countable subset of $X$. If for any $x \in X \setminus A$ the set $f^{-1}(x)$ contains a nontrivial convergent sequence, then $\mathsf{w}(\langle X, \tau\rangle) = \omega$.
\end{lemma}

\begin{proof}
    We prove that $\{f[(a,b)]:a<b \in \mathbb{Q}\}$ is a neighbourhood base at $x$ for all $x \in X \setminus A$. Let $x \in X \setminus A$, $U \in \mathsf{nbhds}(x, \tau)$, $\langle y_n \rangle_{n \in \omega}$ a nontrivial convergent sequence in $\mathbb{S}$ such that $\{y_n : n \in \omega\} \subseteq f^{-1}(x)$, and $y$ its limit point. There exists $c>y$ such that $[y,c) \subseteq f^{-1}[U]$. Since $\langle y_n \rangle_{n \in \omega} \xrightarrow{\mathbb{S}} y$, we see that there exists $n \in \omega$ such that $y_n \in (y, c)$. Take $a<b \in \mathbb{Q}$ such that $y_n \in (a,b) \subseteq (y, c)$, then $x = f(y_n) \in f[(a, b)] \subseteq U$.

    Since $\langle X, \tau \rangle$ is a first-countable space, the Lemma is proved.
\end{proof}

\begin{corollary} \label{Bdiscr}
    Let $\langle X, \tau \rangle$ be a $T_1$ topological space and $f:\mathbb{S}\rightarrow \langle X, \tau \rangle$ a continuous open surjection. If $\mathsf{w}(\langle X, \tau \rangle) > \omega$, then there exists an uncountable set $B \subseteq X$, such that $f^{-1}(x)$ is a closed discrete set for all $x \in B$.
\end{corollary}

\begin{lemma} \label{SembF}
    Let $F$ be an uncountable closed subset of the Sorgenfrey line. Then there exists a closed subset of $F$ that is homeomorphic to the Sorgenfrey line.
\end{lemma}

\begin{proof}
    Since the Sorgenfrey line is hereditary Lindel\"{o}f, we see that there exists an uncountable closed set $C \subseteq F$ such that $C$ has no isolated points. Then from \cite[Theorem 4.3]{8} it follows that $\langle C, \tau_{\mathbb{S}} \upharpoonright C\rangle \cong \mathbb{S}$.
\end{proof}

In the proof of the following theorem we use the ideas from the proof of \cite[Theorem 3.7]{Arkh2}.

\begin{theorem}
    Let $\langle X, \tau \rangle$ be a $T_1$ regular topological space and $f:\mathbb{S}\rightarrow \langle X, \tau \rangle$ a continuous open surjection. If $\mathsf{w}(\langle X, \tau \rangle) > \omega$, then there exists a closed subset of $\langle X, \tau \rangle$ that is homeomorphic to the Sorgenfrey line.
\end{theorem}

\begin{proof}
    Suppose that $\mathsf{w}(\langle X, \tau \rangle) > \omega$. Denote $f \upharpoonright [a,b)$ by $f_{a,b}$ and $\{x \in [a,b) : f^{-1}_{a, b}(f_{a,b}(x)) = \{x\}\}$ by $P_{a, b}$ for all $a<b \in \mathbb{Q}$.

    We prove that
    \begin{equation}\label{Pabclos}
        P_{a, b} \text{ is a closed subset of } \mathbb{S} \text{ for all } a<b \in \mathbb{Q}.
    \end{equation}

    Let $a,b\in \mathbb{Q}$ be such that $a<b$. Since $[a,b)$ is a closed subset of $\mathbb{S}$, we only need to prove that $P_{a,b}$ is a closed subset of $[a,b)_{\mathbb{S}}$. Let $x \in [a,b)$ be such that $x \in \mathsf{Cl}_{\mathbb{S}}(P_{a,b})$. Suppose that $x \not \in P_{a, b}$; then there exists $y \in [a,b)\setminus\{x\}$ such that $f_{a,b}(x) = f_{a,b}(y)$. Fix $U_x\in \mathsf{nbhds}(x, \tau_\mathbb{S} \upharpoonright [a, b))$ and $U_y\in \mathsf{nbhds}(y, \tau_\mathbb{S} \upharpoonright [a, b))$ such that $f[U_x] = f[U_y]$ and $U_x \cap U_y = \emptyset$. Take any point $z \in P_{a,b} \cap U_x$. Since $z\in f[U_x] = f[U_y]$, there exists $z\sp{\prime}\in U_y$ such that $z\sp{\prime} \not= z$ and $f_{a,b}(z) = f_{a,b}(z\sp{\prime})$, it contradicts the fact that $z\in P_{a,b}$.

    Now we prove that
    \begin{equation}\label{fPabhom}
        f \upharpoonright P_{a, b} \text{ is a homeomorphism for all } a<b \in \mathbb{Q}.
    \end{equation}

    Let $a,b\in \mathbb{Q}$ be such that $a<b$. It is easy to see that $f_{a,b}$ is an open map. Since $f^{-1}_{a, b}[f_{a,b}[P_{a, b}]] = P_{a, b}$, we see that $f \upharpoonright P_{a, b} = f_{a,b} \upharpoonright P_{a, b}$ is a bijection and a restriction of the open map to the preimage, and so it is a homeomorphism.

    Let us prove that
    \begin{equation}\label{Pabuncount}
        \text{there exist } a, b \in \mathbb{Q} \text{ such that } a<b \text{ and } |P_{a,b}| > \omega.
    \end{equation}

    Since $\mathsf{w}(\langle X, \tau \rangle) > \omega$, from Corollary \ref{Bdiscr} it follows that there exists an uncountable subset $B \subseteq X$ such that $f^{-1}(x)$ is a closed discrete set for all $x \in B$. It is enough to prove that for any $x \in B$ there exist $a,b\in Q$ such that $a<b$ and $f^{-1}(x) \cap P_{a,b} \not= \emptyset$. Let $x \in B$, since $f^{-1}(x)$ is a closed discrete set, we see that $<$ is a well ordering of this set. Take any isolated point $y \in f^{-1}(x)_\mathbb{R}$. There exist $a,b \in \mathbb{Q}$ such that $a<y<b$ and $[a,b)\cap f^{-1}(x) = \{y\}$, consequently $y \in P_{a,b}$.

    Now we prove that
    \begin{equation}\label{fPabclos}
        f[P_{a,b}] \text{ is a closed subset of } f[[a,b)] \text{ for all } a<b \in \mathbb{Q}.
    \end{equation}

    Let $a,b\in \mathbb{Q}$ be such that $a<b$. Let $x \in f[[a,b)]$ be such that $x \in \mathsf{Cl}_{\langle X, \tau \rangle}(f[P_{a,b}])$. Suppose that $x \not \in f[P_{a,b}]$. Then there exist $y, z \in [a,b)$ such that $y \not= z$ and $f_{a,b}(y) = f_{a,b}(z) = x$. Fix $U_z\in \mathsf{nbhds}(z, \tau_\mathbb{S} \upharpoonright [a, b))$ and $U_y\in \mathsf{nbhds}(y, \tau_\mathbb{S} \upharpoonright [a, b))$ such that $f[U_z] = f[U_y]$ and $U_z \cap U_y = \emptyset$. Since $f$ is an open map, it follows that $f[U_y]\cap f[P_{a,b}] \not = \emptyset$. Take any point $c \in f[U_y]\cap f[P_{a,b}]$. There exist $c_1 \in U_y$ and $c_2 \in U_z$ such that $f_{a,b}(c_1) = f_{a,b}(c_2) = c$, it contradicts the fact that $c \in f[P_{a,b}]$.

     Since $\mathbb{S}$ is a hereditary Lindel\"{o}f space, it follows that $\langle X, \tau \rangle$ is a hereditary Lindel\"{o}f space. Also since $\langle X, \tau \rangle$ is a regular space, it follows that this space is perfectly normal. Fix $a,b \in \mathbb{Q}$ such that $a<b$ and $|P_{a,b}| > \omega$. Take $U := f[[a,b)]$. Since $f$ is an open map, we see that $U \in \tau$. Since $\langle X, \tau \rangle$ is a perfectly normal space, it follows that $U = \bigcup_{i \in \omega}F_i$, where $F_i$ is closed for all $i \in \omega$. From (\ref{fPabhom}) it follows that $f[P_{a,b}]$ is uncountable, and so there exists $n \in \omega$ such that
    \begin{equation}\label{FncfPabunc}
        |F_n \cap f[P_{a,b}]| > \omega.
    \end{equation}

    From (\ref{fPabclos}) it follows that
    \begin{equation}\label{FncfPabclos}
        F_n \cap f[P_{a,b}] \text{ is a closed subset of } \langle X, \tau \rangle.
    \end{equation}

    From (\ref{Pabclos}) and (\ref{fPabhom}) it follows that
    \begin{equation}\label{FncfPabhomS}
        F_n \cap f[P_{a,b}] \text{ is homeomorphic to a closed subset of } \mathbb{S}.
    \end{equation}

    From Lemma \ref{SembF}, (\ref{FncfPabunc}) and (\ref{FncfPabhomS}) it follows that there exists a closed subset of $F_n \cap f[P_{a,b}]$ that is homeomorphic to the Sorgenfrey line. From (\ref{FncfPabclos}) it follows that this set is closed in $\langle X, \tau \rangle$.
\end{proof}

\begin{corollary}
    If a Hausdorff compact space is a continuous open image of the Sorgenrey line, then this space is metrizable.
\end{corollary}

\begin{theorem} \label{crit_comp_metr}
    A Hausdorff compact space is metrizable if and only if it is a continuous open image of the Sorgenfrey line.
\end{theorem}

\begin{proof}
    The theorem follows from the previous Corollary and \cite[Corollary 3.8]{3}.
\end{proof}

{\bf Acknowledgement} The author would like to thank Vova Ivchenko for the help with translation of this paper from Russian to English.The work was performed as part of research conducted in the Ural Mathematical Center.

\bibliographystyle{model1a-num-names}
\bibliography{<your-bib-database>}

\end{document}